\documentclass[11pt]{article}
\usepackage[a4paper, total={6in, 8in}]{geometry}%linia potencjalnie problemowa
\usepackage{amsmath,amssymb, amsthm}
\usepackage{array}
\usepackage{booktabs}
\usepackage{mathtools}
\usepackage{amsthm}
\usepackage{xfrac}
\usepackage{faktor}
\usepackage{stackrel}
\usepackage{multicol}
\usepackage{centernot}
\usepackage{bbm}
\usepackage{tikz}
\usepackage{xcolor}
\usepackage{rotating}
\usepackage{longtable}
\usepackage{tikz-cd}
\usepackage{graphicx}
\usepackage{ulem}
\usepackage{indentfirst}
\usepackage{lipsum}
\usetikzlibrary{arrows}

\newgeometry{tmargin=3cm, bmargin=3cm, lmargin=3cm, rmargin=3cm}

\newcommand\blfootnote[1]{%
  \begingroup
  \renewcommand\thefootnote{}\footnote{#1}%
  \addtocounter{footnote}{-1}%
  \endgroup
}

\def\XXint#1#2#3{{\setbox0=\hbox{$#1{#2#3}{\int}$ }
\vcenter{\hbox{$#2#3$ }}\kern-.6\wd0}}

\theoremstyle{definition}
\newtheorem{thm}{Theorem}
\newtheorem{prop}{Proposition}[section]
\newtheorem{df}[prop]{Definition}
\newtheorem{lem}[prop]{Lemma}

\newtheorem{rem}[prop]{Remark}

\newtheorem*{xrem}{Remark}

\newcommand{\e}{\varepsilon}

\newcommand{\E}{\mathbb{E}}

\newcommand{\R}{\mathbb{R}}

\newcommand{\Prb}{\mathbb{P}}
\newcommand{\Var}{\text{Var}}
\newcommand{\1}{\mathbf{1}}
\newcommand{\vol}{\text{vol}}

\newcommand{\Proj}{\text{Proj}}

\title{Stability of extremal hyperplane projections of balls in $\ell_p^n(\R)$}
\author{Jacek Jakimiuk}
\date{}

\begin{document}
    \maketitle{}

    \begin{abstract}
        We prove stability estimates for the volume of central hyperplane projections of unit balls in $\ell_p^n(\R)$, establishing the dual counterpart of the result of Chasapis, Nayar and Tkocz on stability of sections.

        \blfootnote{\noindent\textit{Key words:} convex functions, Gaussian mixtures, stability, volume of projections.
    
        \textit{2020 Mathematics Subject Classification:} Primary 52A40, 60G50; Secondary 52A20, 60E15.}
    \end{abstract}

    \section{Introduction}

    For $1 \le p \le \infty$, let
    \begin{align*}
        B_p^n=\{x\in\R^n:\|x\|_p\leq 1\}
    \end{align*}
    be the unit ball in the standard $\ell_p^n$ norm. The problem of determining the central hyperplane sections and projections of $B_p^n$ of maximal and minimal volume is a significant problem in convex geometry. Once an extremal hyperplane is known, a natural next question concerns stability, understood as the following quantitative rigidity: how much must the volume deteriorate when the normal vector moves away from the extremiser?

    Let $e_1, \ldots, e_n$ be the standard orthonormal basis in $\R^n$ and let $a_{(k)} = \frac{1}{\sqrt{k}}\sum_{i=1}^ke_i$. It was established in the papers \cite{Ha, He, S, B, MP, K, BN} that $a_{(1)}^{\perp}$ yields the maximal section and projection for $1 \le p < 2$ and the minimal section and projection for $2 < p \le \infty$, $a_{(n)}^{\perp}$ yields the minimal section for $1 \le p < 2$ and the maximal projection for $2 < p \le \infty$, and $a_{(2)}^{\perp}$ yields the maximal section for $p = \infty$ and the minimal projection for $p = 1$. The works \cite{O, ENT-res, Koe} provided partial results in the remaining regime of maximal sections for $2 < p < \infty$ and minimal projections for $1 < p < 2$. The study of stability for sections was developed by Chasapis, Nayar and Tkocz in \cite{CNT}. They obtained dimension-free refinements in all hyperplane-section cases in which the extremisers were known, i.e. in all cases except for maximal sections for $2 < p < \infty$.

    For the stability of projections, much less was known prior to our work. Only for the minimal projection of $B_1^n$ a stability refinement analogous to the result of \cite{CNT} was explicitly stated. De, Diakonikolas and Servedio proved in \cite{DDS} a robust form of Szarek's inequality from \cite{S}: there is a universal $\kappa>0$ such that
    \begin{align*}
        \E\left|\sum_{i=1}^n a_i\e_i\right| \ge \frac{1}{\sqrt2}+\kappa\left|a - a_{(2)}\right|^2,
    \end{align*}
    where $\e_1, \ldots, \e_n$ are i.i.d. Rademacher random variables, i.e. $\Prb(\e_i = 1) = \Prb(\e_i = -1) = \frac 1 2$. As discussed in \cite{ENT-res}, this is equivalent to the stability of the minimal projection of $B_1^n$.

    Our main result is the following theorem, which together with the aforementioned result establishes a dual counterpart of Theorem 1.2 of \cite{CNT}. Here we denote by $\vol_n$ the $n$-dimensional Lebesgue measure and by $\Proj_{a^{\perp}}K$ the projection of a set $K \subseteq \R^n$ onto a hyperplane orthogonal to a vector $a \in \R^n$.

    \begin{thm}\label{main}
        There exists a positive constant $c_p$ depending only on $p$ such that for every $n \ge 1$ and every vector $a = (a_1, \ldots, a_n) \in \R^n$ with $\sum_{i=1}^na_i^2 = 1$ and $a_1 \ge a_2 \ge \ldots \ge a_n \ge 0$ we have:\begin{enumerate}
            \item[(a)] if $2 < p \le \infty$, then
            \begin{align*}
                \frac{\vol_{n-1}\left(\Proj_{a^{\perp}}B_p^n\right)}{\vol_{n-1}\left(\Proj_{a_{(1)}^{\perp}}B_p^n\right)} \ge 1 + c_p\left|a - a_{(1)}\right|^2;
            \end{align*}
            \item[(b)] if $p = \infty$, then
            \begin{align*}
                \frac{\vol_{n-1}\left(\Proj_{a^{\perp}}B_p^n\right)}{\vol_{n-1}\left(\Proj_{a_{(n)}^{\perp}}B_p^n\right)} \le 1 - \frac{c_{\infty}}{n}\sum_{i, j = 1}^n\left(a_i - a_j\right)^2;
            \end{align*}
            \item[(c)] if $2 < p < \infty$, then
            \begin{align*}
                \frac{\vol_{n-1}\left(\Proj_{a^{\perp}}B_p^n\right)}{\vol_{n-1}\left(\Proj_{a_{(n)}^{\perp}}B_p^n\right)} \le 1 - c_p\sum_{i=1}^n\left(a_i^2 - \frac 1 n\right)^2;
            \end{align*}
            \item[(d)] if $1 \le p < 2$, then
            \begin{align*}
                \frac{\vol_{n-1}\left(\Proj_{a^{\perp}}B_p^n\right)}{\vol_{n-1}\left(\Proj_{a_{(1)}^{\perp}}B_p^n\right)} \le 1 - c_p\left|a - a_{(1)}\right|^2.
            \end{align*}
        \end{enumerate}
    \end{thm}

    \begin{xrem}
        Some parts of Theorem \ref{main} ((a), (b), case $p = 1$ in (d)) are easy consequences of previously known or folklore results. However, since (up to our best knowledge) they are nowhere explicitly stated as stability estimates for projections, we decided to include them in our main result. 
    \end{xrem}

    \subsection{Probabilistic representation}

    At the core of our methods is a probabilistic representation of the volume of projections, reducing our problem, in most cases, to the stability of certain moment-comparison inequalities. The most important result for our purposes is the following theorem of Barthe and Naor from \cite{BN}, which states that for $1 < p < \infty$ the formula
    \begin{align}\label{prob-gen}
        \frac{\vol_{n-1}\left(\Proj_{a^{\perp}}B_p^n\right)}{\vol_{n-1}\left(B_p^{n-1}\right)} = \frac{\E\left|\sum_{i=1}^na_iX_i\right|}{\E|X_1|}
    \end{align}
    holds, where $X_1, \ldots, X_n$ are i.i.d. random variables with density proportional to $|t|^{\frac{2-p}{p-1}}e^{-|t|^{\frac{p}{p-1}}}$. The limiting cases $p = 1$ and $p = \infty$ can be obtained by a standard application of Cauchy Projection Formula, which is also the starting point of Barthe's and Naor's proof of \eqref{prob-gen}. This formula states that
    \begin{align}\label{cauchy}
        \vol_{n-1}\left(\Proj_{a^{\perp}}K\right) = \frac 1 2\int_{\mathbb{S}^{n-1}}|\langle a, \theta \rangle|d\sigma_K(\theta),
    \end{align}
    where $\mathbb{S}^{n-1}$ is the unit Euclidean sphere in $\R^n$ and $\sigma_K$ is the surface area measure of $K$. By a direct application of \eqref{cauchy} for $K = B_1^n$ and $K = B_{\infty}^n$ we get
    \begin{align}\label{prob-1}
        \frac{\vol_{n-1}\left(\Proj_{a^{\perp}}B_1^n\right)}{\vol_{n-1}\left(B_1^{n-1}\right)} = \E\left|\sum_{i=1}^na_i\e_i\right|,
    \end{align}
    where $\e_1, \ldots, \e_n$ are i.i.d. Rademacher random variables, and
    \begin{align}\label{prob-inf}
        \frac{\vol_{n-1}\left(\Proj_{a^{\perp}}B_{\infty}^n\right)}{\vol_{n-1}\left(B_{\infty}^{n-1}\right)} = \sum_{i=1}^n|a_i|,
    \end{align}
    respectively.

    \subsection{Notation and organization of the paper}

    By $\E_X$, $\E_Y$, etc. we denote expectation with respect to the random variables $X$ or $X_i$, $Y$ or $Y_i$ etc., using this notation only if we want to emphasize which random variables are fixed and which are randomized. We use the same convention for $\Prb_X$, $\Prb_Y$, etc. The notation may vary across different parts of the paper. In particular, the letters $c$, $C$, $c_p$ and other letters denoting various constants may change their meaning multiple times. We shall not track the explicit constants, but we shall try to make clear on which parameters each constant depends.

    The paper is organized as follows. In Section 2 we write down the proofs of easy parts of Theorem \ref{main}, i.e. parts (a) and (b). As remarked after the formulation of Theorem \ref{main}, the case $p = 1$ of part (d) can also be derived easily from previously known results in the literature, but since it is tied to the general case of part (d), we defer it to Section 4, where all of part (d) is proved. In Section 3 we derive Theorem \ref{main}(c) from a more general result on Gaussian mixtures, which we also prove there. Finally, in Section 5 we give a few remarks comparing parts (b) and (c) of Theorem \ref{main}.

    \section{The easy parts}

    We begin with a very short proof of Theorem \ref{main}(a).

    \begin{proof}[Proof of Theorem \ref{main}(a)]
        Clearly we have $\vol_{n-1}\left(\Proj_{a^{\perp}}B_p^n\right) \ge \vol_{n-1}\left(a^{\perp} \cap B_p^n\right)$ for any vector $a \in \R^n$. Moreover, the equality holds for $a = a_{(1)}$. Thus
        \begin{align*}
            \frac{\vol_{n-1}\left(\Proj_{a^{\perp}}B_p^n\right)}{\vol_{n-1}\left(\Proj_{a_{(1)}^{\perp}}B_p^n\right)} \ge \frac{\vol_{n-1}\left(B_p^n \cap a^{\perp}\right)}{\vol_{n-1}\left(B_p^{n-1}\right)} \ge 1 + c_p\left|a - a_{(1)}\right|^2,
        \end{align*}
        where the second inequality was proved in \cite{CNT}.
    \end{proof}

    Now we shall prove Theorem \ref{main}(b), which after applying formula \eqref{prob-inf} essentially reduces to a variance computation for a suitable random variable. This is a classical argument, but for clarity we give it in full detail.

    \begin{proof}[Proof of Theorem \ref{main}(b)]
        We know by the formula \eqref{prob-inf} that
        \begin{align*}
            \frac{\vol_{n-1}\left(\Proj_{a^{\perp}}B_p^n\right)}{\vol_{n-1}\left(\Proj_{a_{(n)}^{\perp}}B_p^n\right)} = \frac{\sum_{i=1}^na_i}{\sum_{i=1}^n\frac{1}{\sqrt{n}}} = \frac{1}{\sqrt{n}}\sum_{i=1}^na_i,
        \end{align*}
        hence we have
        \begin{align*}
            1 - \frac{\vol_{n-1}\left(\Proj_{a^{\perp}}B_p^n\right)}{\vol_{n-1}\left(\Proj_{a_{(n)}^{\perp}}B_p^n\right)} = 1 - \frac{1}{\sqrt{n}}\sum_{i=1}^na_i = \frac{1 - \frac 1 n\left(\sum_{i=1}^na_i\right)^2}{1 + \frac{1}{\sqrt{n}}\sum_{i=1}^na_i}.
        \end{align*}
        Clearly $1 + \frac{1}{\sqrt{n}}\sum_{i=1}^na_i \le 2$, hence
        \begin{align}\label{simplestep}
            1 - \frac{1}{\sqrt{n}}\sum_{i=1}^na_i \ge \frac 1 2 \left(1 - \frac 1 n\left(\sum_{i=1}^na_i\right)^2\right) = \frac n 2\left(\frac 1 n\sum_{i=1}^na_i^2 - \left(\frac 1 n\sum_{i=1}^na_i\right)^2\right).
        \end{align}
        The RHS of \eqref{simplestep} is equal to $\frac n 2\Var(X)$, where $X$ is a random variable uniformly distributed on the multiset $\{a_1, \ldots, a_n\}$. Let $Y$ be an independent copy of $X$. Then, by a known characterization of the variance we have
        \begin{align*}
            1 - \frac{\vol_{n-1}\left(\Proj_{a^{\perp}}B_p^n\right)}{\vol_{n-1}\left(\Proj_{a_{(n)}^{\perp}}B_p^n\right)} \ge \frac n 2\Var(X) = \frac n 4\E(X - Y)^2 = \frac{1}{4n}\sum_{i, j = 1}^n\left(a_i - a_j\right)^2,
        \end{align*}
        which completes the proof.
    \end{proof}

    \begin{rem}\label{opt}
        The deficit in Theorem \ref{main}(b) is optimal, up to a factor of 2, since $1 \le 1 + \frac{1}{\sqrt{n}}\sum_{i=1}^na_i$ and thus
        \begin{align*}
            1 - \frac 1 n\left(\sum_{i=1}^na_i\right)^2 \ge 1 - \frac{1}{\sqrt{n}}\sum_{i=1}^na_i \ge \frac 1 2 \left(1 - \frac 1 n\left(\sum_{i=1}^na_i\right)^2\right).
        \end{align*}
    \end{rem}

    \section{Proof of Theorem \ref{main}(c)}

    By formula \eqref{prob-gen}, the statement is equivalent to
    \begin{align*}
        \frac{\E\left|\sum_{i=1}^na_iX_i\right|}{\E\left|\frac{1}{\sqrt{n}}\sum_{i=1}^na_iX_i\right|} \le 1 - c_p\sum_{i=1}^n\left(a_i^2 - \frac 1 n\right)^2
    \end{align*}
    or
    \begin{align}\label{c-red1}
        \E\left|\sum_{i=1}^na_iX_i\right| \le \frac{1}{\sqrt{n}}\E\left|\sum_{i=1}^nX_i\right| - \frac{c_p}{\sqrt{n}}\E\left|\sum_{i=1}^nX_i\right|\sum_{i=1}^n\left(a_i^2 - \frac 1 n\right)^2,
    \end{align}
    where $X_i$ are i.i.d. random variables with densities proportional to $|t|^{\frac{2-p}{p-1}}e^{-|t|^{\frac{p}{p-1}}}$. By the Central Limit Theorem and uniform boundedness of the second moments we have that $\lim_{n \to \infty}\frac{1}{\sqrt{n}}\E\left|\sum_{i=1}^nX_i\right| = \E|G|$, where $G$ is a Gaussian random variable with $\E G = 0$ and $\E G^2 = \E X_1^2$. Moreover, the sequence $\frac{1}{\sqrt{n}}\E\left|\sum_{i=1}^nX_i\right|$ is non-decreasing by Theorem 10 in \cite{BN}, which implies that $\frac{1}{\sqrt{n}}\E\left|\sum_{i=1}^nX_i\right| \le \E|G|$. Since $\E|G|$ depends only on $p$, to show the inequality \eqref{c-red1} it suffices to prove that
    \begin{align}\label{c-red2}
        \frac{1}{\sqrt{n}}\E\left|\sum_{i=1}^nX_i\right| - \E\left|\sum_{i=1}^na_iX_i\right| \ge C_p\sum_{i=1}^n\left(a_i^2 - \frac 1 n\right)^2
    \end{align}
    for $C_p = c_p\E|G|$. 

    We shall further reduce \eqref{c-red2} by using several structural properties observed by Eskenazis, Nayar and Tkocz in \cite{ENT-mix}. The first is that the variables $X_i$ are \textit{Gaussian mixtures}.

    \begin{df}
        A random variable $X$ is called a Gaussian mixture if there exists a positive random variable $Y$ and a standard Gaussian random variable $Z$, independent of $Y$, such that $X$ has the same distribution as the product $YZ$.
    \end{df}

    \begin{prop}[proved in \cite{ENT-mix}]\label{mixtura}
        The random variables $X_i$ are Gaussian mixtures.
    \end{prop}

    The second structural property observed in \cite{ENT-mix} and earlier in \cite{AH} is Schur monotonicity of moments of weighted sums of Gaussian mixtures. Recall that a vector $a = (a_1, \ldots, a_n)$ is dominated by a vector $b = (b_1, \ldots, b_n)$ in the Schur order (which is denoted by $a \preceq b$) if
    \begin{align*}
        \sum_{i=1}^na_i = \sum_{i=1}^nb_i \quad \text{and} \quad \forall_{1 \le k < n}\sum_{i=1}^ka_i^* \le \sum_{i=1}^kb_i^*,
    \end{align*}
    where $a_1^*, \ldots, a_n^*$ and $b_1^*, \ldots, b_n^*$ are the non-increasing rearrangements of the coordinates of $a$ and $b$, respectively. Theorem 3 in \cite{ENT-mix} (see also Proposition 2.6 in \cite{AH}) shows that if $X_1, \ldots, X_n$ are i.i.d. Gaussian mixtures, then the function $\left(a_1^2, \ldots, a_n^2\right) \mapsto \E\left|\sum_{i=1}^na_iX_i\right|$ is decreasing with respect to the Schur order (in fact the results of \cite{ENT-mix} and \cite{AH} cover more general moments and functions, not only the first moments, but we shall work only with the first moments). We shall prove the following generalization of this result.

    \begin{thm}\label{mixthm}
        Let $X_1, \ldots, X_n$ be i.i.d. square integrable Gaussian mixtures whose distribution is not Gaussian. Then there exists a positive constant $C$, depending only on the distribution of $X_1$, such that for any real numbers $a_1, \ldots, a_n$, $b_1, \ldots, b_n$ satisfying $\sum_{i=1}^na_i^2 = \sum_{i=1}^nb_i^2 = 1$ and $\left(b_1^2, \ldots, b_n^2\right) \preceq \left(a_1^2, \ldots, a_n^2\right)$ we have
        \begin{align*}
            \E\left|\sum_{i=1}^nb_iX_i\right| - \E\left|\sum_{i=1}^na_iX_i\right| \ge C\left(\sum_{i=1}^na_i^4 - \sum_{i=1}^nb_i^4\right).
        \end{align*}
    \end{thm}

    Observe that since $\sum_{i=1}^n\left(a_i^2 - \frac 1 n\right)^2 = \sum_{i=1}^na_i^4 - \frac 2 n\sum_{i=1}^na_i^2 + \frac 1 n = \sum_{i=1}^na_i^4 - \sum_{i=1}^n\frac{1}{n^2}$, taking $b_i = \frac{1}{\sqrt{n}}$ in Theorem \ref{mixthm} immediately implies \eqref{c-red2}.

    We shall need the following lemma.

    \begin{lem}\label{mixlem}
        Let $Y_1$, $Y_2$ be i.i.d. square integrable random variables whose distribution is not a Dirac delta distribution. Let $a_1$, $a_2$, $b_1$, $b_2$, $s$ also be real numbers such that $s \ge 0$, $a_1^2 + a_2^2 = b_1^2 + b_2^2 = \sigma^2 > 0$ and $a_1^2 \ge b_1^2 \ge b_2^2 \ge a_2^2$. Then there exist positive constants $c$, $\eta$, depending only on the distribution of $Y_1$, such that
        \begin{align*}
            \E\sqrt{b_1^2Y_1^2 + b_2^2Y_2^2 + s} - \E\sqrt{a_1^2Y_1^2 + a_2^2Y_2^2 + s} \ge c\left(s + \eta\sigma^2\right)^{-\frac 3 2}\left(a_1^4 + a_2^4 - b_1^4 - b_2^4\right).
        \end{align*}
    \end{lem}

    \begin{proof}[Proof of Lemma \ref{mixlem}]
        Our idea is to decompose $a_1^2Y_1^2 + a_2^2Y_2^2$ into a symmetric and an antisymmetric part and then use elementary calculus together with some crude bounds. Denote $\Delta \coloneqq \E\sqrt{b_1^2Y_1^2 + b_2^2Y_2^2 + s} - \E\sqrt{a_1^2Y_1^2 + a_2^2Y_2^2 + s}$. Define numbers $\rho \ge \mu \ge \frac 1 2$ such that $a_1^2 = \rho\sigma^2$ and $b_1^2 = \mu\sigma^2$. Define also random variables $R = \frac{Y_1^2 + Y_2^2}{2}$ and $D = Y_1^2 - Y_2^2$. Note that $D$ is a symmetric random variable and hence $D$ has the same distribution as $|D|\e$, where $\e$ is a Rademacher random variable independent of $(Y_1, Y_2)$. We have
        \begin{align*}
            \E\sqrt{b_1^2Y_1^2 + b_2^2Y_2^2 + s} &= \E\sqrt{s + \sigma^2R + \left(\mu - \frac 1 2\right)\sigma^2D} = \E_{\e}\E_{R, D}\sqrt{s + \sigma^2R + \left(\mu - \frac 1 2\right)\sigma^2|D|\e} \\ &= \frac 1 2\E\left[\sqrt{s + \sigma^2R + \left(\mu - \frac 1 2\right)\sigma^2|D|} + \sqrt{s + \sigma^2R - \left(\mu - \frac 1 2\right)\sigma^2|D|}\right]
        \end{align*}
        and in the same manner
        \begin{align*}
            \E\sqrt{a_1^2Y_1^2 + a_2^2Y_2^2 + s} = \frac 1 2\E\left[\sqrt{s + \sigma^2R + \left(\rho - \frac 1 2\right)\sigma^2|D|} + \sqrt{s + \sigma^2R - \left(\rho - \frac 1 2\right)\sigma^2|D|}\right].
        \end{align*}
        Subtracting this and writing the differences of square roots as integrals of the derivatives we get
        \begin{align*}
            \Delta &= \frac 1 4\E\left[\int_{-\sigma^2|D|\left(\rho - \frac 1 2\right)}^{-\sigma^2|D|\left(\mu - \frac 1 2\right)}\frac{dt}{\sqrt{s + \sigma^2R + t}} - \int_{\sigma^2|D|\left(\mu - \frac 1 2\right)}^{\sigma^2|D|\left(\rho - \frac 1 2\right)}\frac{dt}{\sqrt{s + \sigma^2R + t}}\right] \\ &= \frac 1 4\E\int_{\sigma^2|D|\left(\mu - \frac 1 2\right)}^{\sigma^2|D|\left(\rho - \frac 1 2\right)}\left(\frac{1}{\sqrt{s + \sigma^2R - t}} - \frac{1}{\sqrt{s + \sigma^2R + t}}\right)dt \\ &= \frac{\sigma^2}{4}\E|D|\int_{\mu - \frac 1 2}^{\rho - \frac 1 2}\left(\frac{1}{\sqrt{s + \sigma^2R - \sigma^2|D|u}} - \frac{1}{\sqrt{s + \sigma^2R + \sigma^2|D|u}}\right)du \\ &= \frac{\sigma^2}{8}\E|D|\int_{\mu - \frac 1 2}^{\rho - \frac 1 2}\int_{-\sigma^2|D|u}^{\sigma^2|D|u}\left(s + \sigma^2R + v\right)^{-\frac 3 2}dvdu \\ &= \frac{\sigma^4}{8}\E|D|^2\int_{\mu - \frac 1 2}^{\rho - \frac 1 2}\int_{-u}^u\left(s + \sigma^2R + \sigma^2|D|w\right)^{-\frac 3 2}dwdu.
        \end{align*}
        For $-u \le w \le 0$ we may bound $\left(s + \sigma^2R + \sigma^2|D|w\right)^{-\frac 3 2} \ge \left(s + \sigma^2R\right)^{-\frac 3 2}$. Moreover, since the distribution of $Y_1$ is not a Dirac delta distribution and $Y_2$ is an independent copy of $Y_1$, we have that $\Prb(|D| \ge \delta, R \le \eta) > 0$ for some $\delta, \eta > 0$. Thus we bound
        \begin{align*}
            \Delta &\ge \frac{\sigma^4}{8}\E\1_{|D| \ge \delta, R \le \eta}|D|^2\int_{\mu - \frac 1 2}^{\rho - \frac 1 2}\int_{-u}^0\left(s + \sigma^2R + \sigma^2|D|w\right)^{-\frac 3 2}dwdu \\ &\ge \frac{\sigma^4}{8}\int_{\mu - \frac 1 2}^{\rho - \frac 1 2}\int_{-u}^0\E\1_{|D| \ge \delta, R \le \eta}|D|^2\left(s + \sigma^2R\right)^{-\frac 3 2}dwdu \\ &\ge \frac{\sigma^4}{8}\Prb(|D| \ge \delta, R \le \eta)\delta^2\left(s + \eta\sigma^2\right)^{-\frac 3 2}\int_{\mu - \frac 1 2}^{\rho - \frac 1 2}\int_{-u}^01dwdu \\ &= c\left(s + \eta\sigma^2\right)^{-\frac 3 2}\sigma^4\left(\left(\rho - \frac 1 2\right)^2 - \left(\mu - \frac 1 2\right)^2\right) \\ &= c\left(s + \eta\sigma^2\right)^{-\frac 3 2}\sigma^4\frac{\rho^2 + (1 - \rho)^2 - \mu^2 - (1 - \mu)^2}{2} = \frac 1 2c\left(s + \eta\sigma^2\right)^{-\frac 3 2}\left(a_1^4 + a_2^4 - b_1^4 - b_2^4\right),
        \end{align*}
        which completes the proof.
    \end{proof}

    \begin{proof}[Proof of Theorem \ref{mixthm}]
         We start in a manner similar to the proof of Theorem 3 in \cite{ENT-mix}. Let $Y_1, \ldots, Y_n, Z_1, \ldots, Z_n, Z$ be random variables such that $Y_i$ are positive, $Z$, $Z_i$ are standard Gaussians, $X_i$ has the same distribution as $Y_iZ_i$ for all $i$, and variables $X_1, \ldots, X_n$, $Y_1, \ldots, Y_n$, $Z_1, \ldots, Z_n$, $Z$ are independent. Then we have
        \begin{align*}
            \E\left|\sum_{i=1}^na_iX_i\right| = \E\left|\sum_{i=1}^na_iY_iZ_i\right| = \E_Y\E_Z\left|\left(\sum_{i=1}^na_i^2Y_i^2\right)^{\frac 1 2}Z\right| = \E|Z| \cdot \E\sqrt{\sum_{i=1}^na_i^2Y_i^2}.
        \end{align*}
        Since $\E|Z| = \sqrt{\frac{2}{\pi}}$ is a universal constant, it suffices to prove that
        \begin{align*}
            \E\sqrt{\sum_{i=1}^nb_i^2Y_i^2} - \E\sqrt{\sum_{i=1}^na_i^2Y_i^2} \ge C\left(\sum_{i=1}^na_i^4 - \sum_{i=1}^nb_i^4\right).
        \end{align*}
        By a known characterization of the Schur order we have $\left(b_1^2, \ldots, b_n^2\right) = b \preceq a = \left(a_1^2, \ldots, a_n^2\right)$ if and only if $b$ can be obtained from $a$ through a composition of finitely many $T$-transformations, where a $T$-transformation is a transformation nearing two coefficients of a vector while preserving their sum and all remaining coefficients. Denote these $T$-transformations by $T_1, \ldots, T_k$, so that $b = T_kT_{k-1} \ldots T_1a$, and denote $T^j = T_jT_{j-1} \ldots T_1$ for $j = 1, 2, \ldots, k$. Then, by a telescoping sum argument, it suffices to prove that
        \begin{align}\label{tt}
            \E\sqrt{\sum_{i=1}^n\left(T^ja\right)_iY_i^2} - \E\sqrt{\sum_{i=1}^n\left(T^{j-1}a\right)_iY_i^2} \ge C\left(\sum_{i=1}^n\left(T^{j-1}a\right)_i^2- \sum_{i=1}^n\left(T^ja\right)_i^2\right)
        \end{align}
        for $j = 1, \ldots, k$. By invariance under permutations and the structure of $T$-transformations, in order to simplify the notation we shall slightly abuse it and write that $T^{j-1}a = \left(a_1^2, \ldots, a_n^2\right)$, $T^ja = \left(b_1^2, b_2^2, a_3^2, \ldots, a_n^2\right)$ and $S = \sum_{i=3}^na_i^2Y_i^2$. Then \eqref{tt} becomes
        \begin{align*}
            \E\sqrt{b_1^2Y_1^2 + b_2^2Y_2^2 + S} - \E\sqrt{a_1^2Y_1^2 + a_2^2Y_2^2 + S} \ge C\left(a_1^4 + a_2^4 - b_1^4 - b_2^4\right).
        \end{align*}
        By Lemma \ref{mixlem} we have
        \begin{align*}
            \E\sqrt{b_1^2Y_1^2 + b_2^2Y_2^2 + S} - \E\sqrt{a_1^2Y_1^2 + a_2^2Y_2^2 + S} \ge c\left(a_1^4 + a_2^4 - b_1^4 - b_2^4\right)\E\left(S + \eta\sigma^2\right)^{-\frac 3 2},
        \end{align*}
        where $\sigma^2 = a_1^2 + a_2^2 = b_1^2 + b_2^2$ and $c$, $\eta$ are positive constants depending only on the distribution of $Y_1$ (and hence only on the distribution of $X_1$). By the Markov inequality we know that
        \begin{align*}
            \Prb\left(S \le 2\E Y_1^2\right) = 1 - \Prb\left(S > 2\E Y_1^2\right) \ge 1 - \frac{\E S}{2\E Y_1^2} = 1 - \frac{1 - \sigma^2}{2} \ge \frac 1 2
        \end{align*}
        and hence
        \begin{align*}
            \E\left(S + \eta\sigma^2\right)^{-\frac 3 2} \ge \Prb\left(S \le 2\E Y_1^2\right)\left(2\E Y_1^2 + \eta\right)^{-\frac 3 2}
        \end{align*}
        since $\sigma^2 \le 1$. As the RHS depends only on the distribution of $Y_1$, this completes the proof.
    \end{proof}

    \section{Proof of Theorem \ref{main}(d)}

    We shall begin by proving the case $p = 1$. Below we present our own proof. In the remark after it we explain how the case $p = 1$ can be derived much easier and why we decided to take a more difficult approach.
    
    \begin{proof}[Proof of the case $p = 1$]
        By formula \eqref{prob-1}, together with the observation that 
        \begin{align*}
            \left|a - a_{(1)}\right|^2 = (1 - a_1)^2 + \sum_{i=2}^na_i^2 = 2 - 2a_1,
        \end{align*} 
        our statement reduces to
        \begin{align}\label{d1-red1}
            \E\left|\sum_{i=1}^na_i\e_i\right| \le 1 - c(1 - a_1)
        \end{align}
        for some universal constant $c$, where $\e_1, \ldots, \e_n$ are i.i.d. Rademacher random variables. Now we split the proof into two cases.

        \medskip

        \noindent \textit{Case 1.} $a_1 \ge c'$, where $c' > 0$ is a constant to be chosen later. Let $G_1, \ldots, G_n$ be i.i.d. Gaussian random variables with mean 0 and variance $\frac{\pi}{2}$, so that $\E|G_i| = 1$ and $G_i$ has density $\frac{1}{\pi}e^{-\frac{x^2}{\pi}}$. Let the variables $G_i$ also be independent of the variables $\e_i$. Then \eqref{d1-red1} can be written as
        \begin{align}\label{d1-red2}
            \E\left|\sum_{i=1}^na_iG_i\right| - \E\left|\sum_{i=1}^na_i\e_i\right| \ge c(1 - a_1).
        \end{align}
        Now our strategy is as follows. We want to use the Lindeberg swapping argument, i.e. exchange Gaussians for Rademachers one by one, keeping track of the deficit. For an example of how stability estimates can be derived in this way, see the proof of Theorem 1 in \cite{J}. However, unlike in \cite{J}, we want to derive the full deficit on the first exchange and only guarantee that later exchanges do not decrease the deficit. An advantage of this modification is that we may guarantee that the part untouched by an exchange is a Gaussian random variable, which simplifies estimates depending on its distribution.
    
        Let $Z = \sum_{i=2}^na_iG_i$. Using the Lindeberg swapping argument, we see that
        \begin{align*}
            \E|a_1\e_1 + Z| - \E\left|\sum_{i=1}^na_i\e_i\right| = \sum_{i=2}^n\left(\E\left|\sum_{j < i}a_j\e_j + \sum_{j \ge i}a_jG_j\right| - \E\left|\sum_{j \le i}a_j\e_j + \sum_{j > i}a_jG_j\right|\right).
        \end{align*}
        We claim that all terms in the sum on the RHS above are non-negative. Indeed, in the $i$-th term we may condition on all variables with indices $j \ne i$ and observe, by Jensen's inequality, that for every $t \in \R$ we have
        \begin{align*}
            \E|a_iG_i + t| = \E|a_i|G_i|\e_i + t| \ge \E|a_i\e_i\E|G_i| + t| = \E|a_i\e_i + t|.
        \end{align*}
        Thus, to show \eqref{d1-red2} it suffices to prove that
        \begin{align}\label{d1-finalred}
            \E|a_1G_1 + Z| - \E|a_1\e_1 + Z| \ge c(1 - a_1).
        \end{align}
        Let $\phi_z(x) = \E|x\e_1 + z| = \frac 1 2|x + z| + \frac 1 2|x - z| = \max\{|x|, |z|\}$ for $x, z \in \R$. Then
        \begin{align*}
            \E|a_1G_1 + Z| - \E|a_1\e_1 + Z| = \E[\phi_Z(a_1|G_1|) - \phi_Z(a_1)].
        \end{align*}
        Since
        \begin{align*}
            \E\1_{|Z| > a_1}(\phi_Z(a_1|G_1|) - \phi_Z(a_1)) &= \E\1_{|Z| > a_1}(\max\{|Z|, a_1|G_1|\} - |Z|) \\ &= \E\1_{|Z| > a_1}(a_1|G_1| - |Z|)_+ \ge 0,
        \end{align*}
        we have (using the independence and Gaussianity of $G_1$ and $Z$ together with the observation that $Z$ has the same distribution as $\sqrt{1-a_1^2}G_1$)
        \begin{align*}
            \E[\phi_Z(a_1|G_1|) - \phi_Z(a_1)] &\ge \E\1_{|Z| \le a_1}(\phi_Z(a_1|G_1|) - \phi_Z(a_1)) = \E\1_{|Z| \le a_1}(\max\{|Z|, a_1|G_1|\} - a_1) \\ &= \E\1_{|Z| \le a_1}\left((|Z| - a_1|G_1|)_+ + a_1|G_1| - a_1\right) \\ &= \E\1_{|Z| \le a_1}\E[a_1|G_1| - a_1] + \E\1_{a_1|G_1| \le |Z| \le a_1}(|Z| - a_1|G_1|) \\ &= \E_Z\left[\1_{|Z| \le a_1}\E_G\1_{a_1|G_1| \le |Z|}(|Z| - a_1|G_1|)\right] \\ &\ge \E_Z\left[\1_{|Z| \le a_1}\E_G\1_{a_1|G_1| \le |Z|/2}(|Z| - a_1|G_1|)\right] \\ &\ge \E\left[\1_{|Z| \le a_1}\Prb_G\left(|G_1| \le \frac{|Z|}{2a_1}\right)\frac{|Z|}{2}\right] \ge \E\1_{|Z| \le a_1}\frac{|Z|}{2} \cdot \frac{|Z|}{a_1} \cdot \frac{1}{\pi}e^{-\frac{|Z|^2}{4\pi a_1^2}} \\ &\ge \frac{1}{2\pi a_1}e^{-\frac{a_1^2}{4\pi a_1^2}}\E\1_{|Z| \le a_1}|Z|^2 = \frac{\tilde{c}\left(1 - a_1^2\right)}{a_1}\E\1_{|G_1|^2 \le \frac{a_1^2}{1 - a_1^2}}|G_1|^2 \\ &\ge \tilde{c}(1 - a_1)\E\1_{|G_1|^2 \le (c')^2}|G_1|^2 = c(1 - a_1).
        \end{align*}
        Hence \eqref{d1-finalred} is proved.

        \medskip

        \noindent \textit{Case 2.} $a_1 < c'$. Then to prove \eqref{d1-red1} it suffices to show that
        \begin{align}\label{d1-altred}
            \E\left|\sum_{i=1}^na_i\e_i\right| \le 1 - c.
        \end{align}
        Let $G$ be a standard Gaussian random variable with mean 0 and variance 1. Then we have $\E|G| = \sqrt{\frac{2}{\pi}} < 1$. By the Central Limit Theorem, one expects that for suitably small $c'$ the LHS of \eqref{d1-altred} is close to $\E|G|$ and hence bounded away from 1. It remains to formalize this idea.

        Denote $S = \sum_{i=1}^na_i\e_i$. We have
        \begin{align*}
            |\E|S| - \E|G|| &= \left|\int_0^{\infty}\Prb(|S| > t)dt - \int_0^{\infty}\Prb(|G| > t)dt\right| \\ &\le \int_0^T|\Prb(|S| > t) - \Prb(|G| > t)|dt + \int_T^{\infty}(\Prb(|S| > t) + \Prb(|G| > t))dt
        \end{align*}
        for any $T > 0$. By standard tail bounds for Gaussian and Rademacher random variables the second integral is bounded by $2\int_T^{\infty}e^{-\frac{t^2}{2}}dt \le 2\int_T^{\infty}e^{1-t}dt = 2e^{1-T}$. By the Berry-Esseen bound the first integral is bounded by $CT\sum_{i=1}^na_i^3 \le CTa_1\sum_{i=1}^na_i^2 \le CTc'$, where $C$ is the constant from the Berry-Esseen theorem. Thus if we choose $T$ such that $2e^{1-T} < \frac 1 4\left(1 - \sqrt{\frac{2}{\pi}}\right)$, then $c'$ such that $CTc' < \frac 1 4\left(1 - \sqrt{\frac{2}{\pi}}\right)$, and finally take $c \le \frac 1 2\left(1 - \sqrt{\frac{2}{\pi}}\right)$, then
        \begin{align*}
            \E|S| \le \E|G| + \frac 1 2\left(1 - \sqrt{\frac{2}{\pi}}\right) \le 1 - c.
        \end{align*}
        This completes the proof.
    \end{proof}

    \begin{xrem}
        After writing down this proof we learned (with the help of ChatGPT 5.6 Sol (Plus), see \textbf{AI tools disclosure} at the end of the paper) that the inequality \eqref{d1-red1} follows easily from Theorem 1.3 in \cite{JOW}, which is a strengthening of the ideas developed in \cite{FKN}. Indeed, it suffices to substitute $X_i = a_i\e_i$ in Theorem 1.3 in \cite{JOW} and do a simple computation. However, the result of \cite{JOW} does not imply the analogous result for $p > 1$, since it compares the absolute first moment of a sum with its variance, whereas we compare it with the first absolute moment of a single random variable. These two quantities coincide in the Rademacher case, but not for the random variables $X_i$ from the formula \eqref{prob-gen}. We decided to keep our proof in the paper since we shall proceed with the case $1 < p < 2$ in an analogous way and we think that the main idea of our approach is easier to understand for $p = 1$.
    \end{xrem}

    For $p > 1$ our strategy is roughly the same as for $p = 1$. However, in the aforementioned case we heavily relied on the decomposition $G \sim |G|\e$, where $G$, $\e$ are Gaussian and Rademacher random variables, respectively. This decomposition has no obvious counterpart for random variables with density proportional to $|t|^{\frac{2-p}{p-1}}e^{-|t|^{\frac{p}{p-1}}}$ in place of Rademacher random variables. In order to overcome this difficulty we shall follow the ideas of \cite{BN} and use the so-called Choquet ordering. It can be defined for arbitrary Radon measures on $\R^n$, but since we are interested only in symmetric random variables, we shall restrict our definition to this case.

    \begin{df}
        Let $X$ and $Y$ be symmetric random variables. We say that $X$ is dominated by $Y$ in the Choquet ordering and denote $X \preceq Y$ if for every convex function $\phi \colon \R \to [0, \infty)$ we have
        \begin{align*}
            \E\phi(X) \le \E\phi(Y).
        \end{align*}
    \end{df}

    The following proposition was proved by Barthe and Naor in \cite{BN}.

    \begin{prop}\label{choquet}
        Let $1 < p < 2$, let $G$ be a Gaussian random variable with mean 0 and variance $\frac{\pi}{2}$, so that $\E|G| = 1$, and $X$ be a random variable with density $f_p(t) = \alpha_p|t|^{\frac{2-p}{p-1}}e^{-\beta_p|t|^{\frac{p}{p-1}}}$, where $\alpha_p$ and $\beta_p$ are such that $f_p$ is a probability density and $\E|X| = 1$. Then $X \preceq G$.
    \end{prop}

    \begin{proof}[Proof of Theorem \ref{main}(d) for $1 < p < 2$]
        As in the case $p = 1$, using formula \eqref{prob-gen} instead of \eqref{prob-1}, we reduce the statement to
        \begin{align}\label{d-red1}
            \E\left|\sum_{i=1}^na_iX_i\right| \le 1 - c_p(1 - a_1)
        \end{align}
        for some constant $c_p$ depending only on $p$, where $X_1, \ldots, X_n$ are independent copies of $X$ from Proposition \ref{choquet}. We again split the proof into two cases.

        \medskip

        \noindent \textit{Case 1.} $a_1 \ge c_p'$, where $c_p' > 0$ is a constant to be chosen later. Let $G_1, \ldots, G_n$ be independent copies of $G$ from Proposition \ref{choquet}, independent of the variables $X_i$. Clearly \eqref{d-red1} reduces to
        \begin{align*}
            \E\left|\sum_{i=1}^na_iG_i\right| - \E\left|\sum_{i=1}^na_iX_i\right| \ge c_p(1 - a_1).
        \end{align*}
        Following the strategy used in the case $p = 1$, we denote $Z = \sum_{i=2}^na_iG_i$ and observe that, by Proposition \ref{choquet} and convexity of $x \mapsto |x + t|$, we have $\E|a_iG_i + t| \ge \E|a_iX_i + t|$ for every $t \in \R$. Hence it suffices to prove that
        \begin{align}\label{d-finalred}
            \E|a_1G_1 + Z| - \E|a_1X_1 + Z| \ge c_p(1 - a_1).
        \end{align}
        Let $f_p$, $g$ be the densities of $X_1$, $G_1$, respectively, i.e. $f_p$ is the same as in Proposition \ref{choquet} and $g(x) = \frac{1}{\pi}e^{-\frac{x^2}{\pi}}$. Let also $\phi_z(x) = \frac 1 2|x + z| + \frac 1 2|x - z| = \max\{|x|, |z|\}$ for $x, z \in \R$. Then
        \begin{align}\label{d-pom}
            \E|a_1G_1 + Z| - \E|a_1X_1 + Z| = 2\E\int_0^{\infty}\phi_Z(a_1x)(g(x) - f_p(x))dx.
        \end{align}
        Fix $z \in \R$ for a moment. Since $\int_0^{\infty}x(g(x) - f_p(x))dx = 0$ (recall that $\E|X_1| = \E|G_1|$), we have
        \begin{align}\label{d-int1}
            \int_0^{\infty}\phi_z(a_1x)(g(x) - f_p(x))dx &= \int_0^{\infty}(a_1x + (|z| - a_1x)_+)(g(x) - f_p(x))dx \notag \\ &= \int_0^{\frac{|z|}{a_1}}(|z| - a_1x)(g(x) - f_p(x))dx.
        \end{align}
        By Lemma 9 in \cite{BN} we know that there exists numbers $0 < x_p < y_p < \infty$ such that $g(x) > f_p(x)$ for $x \in [0, x_p) \cup (y_p, \infty)$ and $g(x) < f_p(x)$ for $x \in (x_p, y_p)$. In particular, there exist positive constants $b_p$, $d_p$, depending only on $p$, such that $g(x) - f_p(x) \ge b_p$ for $0 \le x \le d_p$. Thus, if we additionally assume that $|z| \le a_1x_p$, do that $g(x) \ge f_p(x)$ for $0 \le x \le \frac{|z|}{a_1}$, we get
        \begin{align}\label{d-int2}
            \int_0^{\frac{|z|}{a_1}}(|z| - a_1x)(g(x) - f_p(x))dx &\ge \int_0^{\min\left\{d_p, \frac{|z|}{2a_1}\right\}}(|z| - a_1x)(g(x) - f_p(x))dx \notag \\ &\ge \frac{b_p|z|}{2}\min\left\{d_p, \frac{|z|}{2a_1}\right\}.
        \end{align}
        Combining \eqref{d-pom}, \eqref{d-int1} and \eqref{d-int2} leads to
        \begin{align*}
            \E|a_1G_1 + Z| - \E|a_1X_1 + Z| \ge b_p\E|Z|\1_{|Z| \le a_1x_p}\min\left\{d_p, \frac{|Z|}{2a_1}\right\} \ge \frac{b_p}{2a_1}\E|Z|^2\1_{|Z| \le a_1\min\{x_p, 2d_p\}}.
        \end{align*}
        The first inequality above holds true since the integral $\int_0^{\infty}\phi_z(a_1x)(g(x) - f_p(x))dx$ is non-negative for every $z \in \R$ due to the convexity of $\phi_z$ and Proposition \ref{choquet}. 
        
        Now we finish the proof of \eqref{d-finalred} in a manner similar to the last three lines of the corresponding case $a_1 \ge c'$ for $p = 1$.

        \medskip

        \noindent \textit{Case 2.} $a_1 < c_p'$. To proceed completely analogously to the case $p = 1$ we only need to ensure that $\E X_1^2 < \frac{\pi}{2}$ and that $\sum_{i=1}^na_iX_i$ admits suitable tail estimates. For the latter, observe that the random variable $X_1$ is $\frac{\pi}{2}$-subgaussian. It follows immediately from Proposition \ref{choquet}, since by convexity of $x \mapsto e^{\lambda x}$ we have $\E e^{\lambda X_1} \le \E e^{\lambda G_1} \le e^{\frac{\pi\lambda^2}{4}}$ for all $\lambda \in \R$. Thus the random variable $\sum_{i=1}^na_iX_i$ is also $\frac{\pi}{2}$-subgaussian and hence admits subgaussian tail estimates, allowing to bound $\int_T^{\infty}\Prb\left(\left|\sum_{i=1}^na_iX_i\right| > t\right)dt$ exponentially in $T$.

        Denote $\E X_1^2 = \sigma^2$ and recall that the density $f_p$ of $X_1$ is given by $\alpha_p|x|^{\frac{2-p}{p-1}}e^{-\beta_p\frac{p}{p-1}}$. By standard computations we deduce that
        \begin{align*}
            \int_{\R}f_p(x)dx = 1 = \int_{\R}|x|f_p(x)dx \quad \Longrightarrow \frac{2(p-1)\alpha_p\Gamma\left(\frac 1 p\right)}{p\beta_p^{\frac 1 p}} = 1 = \frac{2(p-1)\alpha_p}{p\beta_p}.
        \end{align*}
        Hence $\beta_p = \Gamma\left(\frac 1 p\right)^{-\frac{p}{p-1}}$ and
        \begin{align*}
            \sigma^2 &= \int_{\R}x^2f_p(x)dx = \frac{2(p-1)\alpha_p\Gamma\left(2 - \frac 1 p\right)}{p\beta_p^{2 - \frac 1 p}} = \frac{\Gamma\left(2 - \frac 1 p\right)}{\beta_p^{1 - \frac 1 p}} = \Gamma\left(2 - \frac 1 p\right)\Gamma\left(\frac 1 p\right) = \frac{\pi\left(1 - \frac 1 p\right)}{\sin\left(\frac{\pi}{p}\right)}.
        \end{align*}
        Since the function $t \mapsto \frac{t}{\sin(\pi - t)} = \frac{t}{\sin t}$ is increasing on $\left(0, \frac{\pi}{2}\right)$, we have $\sigma^2 < \frac{\pi}{2}$. This completes the proof of Theorem \ref{main}(d).
    \end{proof}

    \section{Final remarks}

    One could expect that deficits in Theorem \ref{main}(b) and (c) to have similar forms, since the unit ball in $\ell_{\infty}^n$ is, in some sense, a limit of the unit balls in $\ell_p^n$ as $p \to \infty$. However, the actual forms of these deficits are hard to compare. In order to do so, we shall prove the following proposition.

    \begin{prop}\label{b-altform}
        There exists a positive constant $c_{\infty}'$ such that for every $n \ge 1$ and every vector $a = (a_1, \ldots, a_n) \in \R^n$ with $\sum_{i=1}^na_i^2 = 1$ and $a_1 \ge a_2 \ge \ldots \ge a_n \ge 0$ we have
        \begin{align*}
            \frac{\vol_{n-1}\left(\Proj_{a^{\perp}}B_{\infty}^n\right)}{\vol_{n-1}\left(\Proj_{a_{(n)}^{\perp}}B_{\infty}^n\right)} \le 1 - \frac{c'_{\infty}}{a_1^3\sqrt{n}}\sum_{i=1}^n\left(a_i^2 - \frac 1 n\right)^2.
        \end{align*}
    \end{prop}

    We shall need the following simple lemma.

    \begin{lem}\label{altbound}
        Let $0 \le x < y \le 1$. Then
        \begin{align*}
            \int_x^y\frac{y - t}{t^{\frac 3 2}}dt \ge \int_x^y\frac{t - x}{t^{\frac 3 2}}dt \ge \frac{(y-x)^2}{4y^{\frac 3 2}}.
        \end{align*}
    \end{lem}

    \begin{proof}[Proof of Lemma \ref{altbound}]
        The first inequality follows easily from the computation
        \begin{align*}
            \int_x^y\frac{y + x - 2t}{t^{\frac 3 2}}dt = \int_x^{\frac{x+y}{2}}(x + y - 2t)\left(t^{-\frac 3 2} - ((x + y - t)^{-\frac 3 2}\right)dt \ge 0.
        \end{align*}
        For the second inequality, we crudely estimate
        \begin{align*}
            \int_x^y\frac{t - x}{t^{\frac 3 2}}dt \ge \int_{\frac{x+y}{2}}^y\frac{t - x}{t^{\frac 3 2}}dt \ge \frac{y - x}{2} \cdot \frac{\frac{y - x}{2}}{y^{\frac 3 2}} =  \frac{(y-x)^2}{4y^{\frac 3 2}}.
        \end{align*}
    \end{proof}

    \begin{proof}[Proof of Proposition \ref{b-altform}]
        Recall that
        \begin{align*}
            \frac{\vol_{n-1}\left(\Proj_{a^{\perp}}B_p^n\right)}{\vol_{n-1}\left(\Proj_{a_{(n)}^{\perp}}B_p^n\right)} = \frac{1}{\sqrt{n}}\sum_{i=1}^na_i.
        \end{align*}
        Denote $x_i = a_i^2$. Using the Taylor expansion with integral remainder of the function $x \mapsto \sqrt{x}$, we have
        \begin{align*}
            1 - \frac{\vol_{n-1}\left(\Proj_{a^{\perp}}B_p^n\right)}{\vol_{n-1}\left(\Proj_{a_{(n)}^{\perp}}B_p^n\right)} &= 1 - \frac{1}{\sqrt{n}}\sum_{i=1}^n\sqrt{x_i} = \frac{1}{\sqrt{n}}\sum_{i=1}^n\left(\frac{1}{\sqrt{n}} - \sqrt{x_i}\right) \\ &= \frac{1}{\sqrt{n}}\sum_{i=1}^n\left(-\frac{\sqrt{n}}{2}\left(x_i - \frac 1 n\right) - \int_{\frac 1 n}^{x_i}-\frac{1}{4t^{\frac 3 2}}(x_i - t)dt\right) \\ &= \frac{1}{4\sqrt{n}}\sum_{i=1}^n\int_{\frac 1 n}^{x_i}\frac{x_i - t}{t^{\frac 3 2}}dt,
        \end{align*}
        since $\sum_{i=1}^n\left(x_i - \frac 1 n\right) = 0$. We use the convention $\int_x^y = -\int_y^x$ and hence $\int_{\frac 1 n}^{x_i}\frac{x_i - t}{t^{\frac 3 2}}dt = \int_{x_i}^{\frac 1 n}\frac{t - x_i}{t^{\frac 3 2}}dt$ if $x_i < \frac 1 n$. Thus, by Lemma \ref{altbound}, we have
        \begin{align*}
            \frac{1}{4\sqrt{n}}\sum_{i=1}^n\int_{\frac 1 n}^{x_i}\frac{x_i - t}{t^{\frac 3 2}}dt \ge \frac{1}{16\sqrt{n}}\sum_{i=1}^n\frac{\left(x_i - \frac 1 n\right)^2}{\max\left\{\frac 1 n, x_i\right\}^{\frac 3 2}} \ge \frac{c'_{\infty}}{a_1^3\sqrt{n}}\sum_{i=1}^n\left(a_i^2 - \frac 1 n\right)^2,
        \end{align*}
        where in the last inequality we use $a_1 \ge \max\left\{a_i, \frac{1}{\sqrt{n}}\right\}$ for all $i$. This completes the proof.
    \end{proof}

    \begin{rem}
        If a vector $a$ is close to the vector $a_{(n)}$, in particular $a_1$ is of order $\frac{1}{\sqrt{n}}$, then the deficit obtained in Proposition \ref{b-altform} is better than the deficit in Theorem \ref{main}(c) by a factor of order $n$ (omitting the dependence on $p$ and assuming that $c_p \sim c_{\infty}$). Clearly it is no better than the deficit in Theorem \ref{main}(b), as pointed out in Remark \ref{opt}.
    \end{rem}

    \begin{rem}
        Let $X$ be the random variable from the proof of Theorem \ref{main}(b). We have
        \begin{align*}
            \Var(X) = \frac{1}{2n^2}\sum_{i, j = 1}^n\left(a_i - a_j\right)^2.
        \end{align*}
        On the other hand, we have also
        \begin{align*}
            \Var\left(X^2\right) = \frac 1 n\sum_{i=1}^n\left(a_i^2 - \frac 1 n\right)^2.
        \end{align*}
        Together with Theorem \ref{main}(b), Proposition \ref{b-altform} and Remark \ref{opt} this yields
        \begin{align*}
            \Var(X)\|X\|_{\infty}^3 \gtrsim \frac{1}{\sqrt{n}}\Var\left(X^2\right) = \Var\left(X^2\right)\sqrt{\E X^2}.
        \end{align*}
        After omitting the middle expression, the resulting inequality is homogeneous and thus holds for any non-negative random variable $X$ uniformly distributed on an $n$-element multiset.
    \end{rem}

    \medskip

    \noindent \textbf{AI tools disclosure.} Large language model ChatGPT 5.6 Sol (Plus) was used to generate parts of the text, improve the language of the paper, check typos and search the literature. In particular ChatGPT found the references \cite{JOW}, \cite{FKN} and pointed a few minor mathematical errors. All mathematical ideas, as well as their technical implementation, were developed entirely by the author. The author carefully checked all AI-generated content and takes full responsibility for it.

    \medskip
    \medskip

    \noindent \textbf{Acknowledgement.} I would like to thank Piotr Nayar for helpful discussions.

    \medskip
    \medskip

    \noindent Institute of Mathematics \\
    University of Warsaw \\
    02-097, Warsaw, Poland \\
    jj406165@mimuw.edu.pl
	
\end{document}